%% file: article.tex
\documentclass[a4paper,12pt]{amsart}

\usepackage[latin1]{inputenc}
\usepackage[T1]{fontenc}
\usepackage[linktocpage]{hyperref} 
\hypersetup{colorlinks=true,linkcolor=blue,citecolor=blue,filecolor=magenta,urlcolor=blue}  
\usepackage{geometry}      
\usepackage[english]{babel}  
\usepackage{amsmath,amsfonts,amssymb,amsthm,mathrsfs}
\usepackage{dsfont}
\usepackage{tikz,tkz-graph,tikz-cd}
\usepackage{enumerate}

\newtheorem{thm}{Theorem}[section]

\newtheorem{lem}[thm]{Lemma}
\newtheorem{cor}[thm]{Corollary}

\theoremstyle{definition}

\theoremstyle{remark}
\newtheorem{rmk}[thm]{Remark}

\newenvironment{question}{\par\vspace{5pt}\noindent\textbf{Question.}\itshape}{ \par\vspace{5pt}}

\newcommand{\CC}{\mathds{C}}

\newcommand{\FF}{\mathds{F}}
\newcommand{\PP}{\mathds{P}}

\newcommand{\A}{\mathcal{A}}
\newcommand{\B}{\mathcal{B}}

\newcommand{\M}{\mathcal{M}}
\newcommand{\N}{\mathcal{N}}

\newcommand{\R}{\mathcal{R}}

\newcommand{\mfM}{\mathfrak{M}}
\newcommand{\mfN}{\mathfrak{N}}
\renewcommand{\S}{\mathcal{S}}

\renewcommand{\L}{\mathcal{L}}

\DeclareMathOperator{\Sing}{Sing}
\DeclareMathOperator{\PGL}{PGL}

\begin{document}


\title[Topology and homotopy of lattice isomorphic arrangements]{Topology and homotopy of lattice isomorphic arrangements}

\author{Beno\^it Guerville-Ball\'e}
\address{
Instituto de Ci\^encias Matem\'aticas e de Computa\c{c}\~ao
Universidade de S\~ao Paulo
Avenida Trabalhador Sancarlense, 400 - Centro
S\~ao Carlos - SP, 13566-590 (Brazil)
}
\email{benoit.guerville-balle@math.cnrs.fr}
\thanks{During the current work the author has been supported by a JSPS post-doctoral grant and by the postdoctoral grant \#2017/15369-0 of the Funda\c{c}\~ao de Amparo \`a Pesquisa do Estado de S\~ao Paulo (FAPESP)}
  
  \subjclass[2010]{
52C30, 
32S22, 
32Q55, 
54F65, 
14E25 
}		

\begin{abstract}
  We prove the existence of lattice isomorphic line arrangements having $\pi_1$-equivalent or homotopy-equivalent complements and non homeomorphic embeddings in the complex projective plane. We also provide two explicit examples, one is formed by real-complexified arrangements while the second is not.
\end{abstract}

\maketitle


\section*{Introduction}
	\input{introduction}

\section{Existence of homotopy-equivalent Zariski pair}\label{sec:ZP}
	\input{ZP}

\section{Preserving the homotopy-equivalence in two known examples}\label{sec:example}
	\input{example}

\section*{Acknowledgement}

The author would like to thank \textsc{J.~Viu-Sos} for all the rewarding discussions and for his helpful remarks on this manuscript.


\bibliographystyle{plain}
\bibliography{biblio}

\end{document}

%% file: introduction.tex
A \emph{line arrangement} $\A$ is a finite collection $\{L_1,\dots,L_n\}$ of lines in the complex projective plane $\CC\PP^2$. Its \emph{topology} is defined as the homeomorphism type of its embedding in $\CC\PP^2$. The complement 
$M(\A)=\CC\PP^2\setminus \bigcup_{L\in\A}L$ is an important invariant of the topology. The two main results about it seem to be antithetic. The first is due to Orlik and Solomon in~\cite{OrlSol}, where they prove that the cohomology ring of the complement is determined by the \emph{intersection lattice}. The latter, the one of Rybnikov~\cite{Ryb}, asserts that the fundamental group of the complement is not determined by the intersection lattice, providing thus the first example of lattice isomorphic arrangements having different topologies (also called a \emph{Zariski pair}).

It is known that the fundamental group of $M(\A)$ is not determined by the topology. Indeed, Falk constructs in~\cite{Falk}, an explicit example of two arrangements having homotopy-equivalent complements and non-homeomorphic topologies. Nevertheless, these arrangements are not lattice isomorphic. However, the result of Rybnikov~\cite{Ryb} (see also~\cite{ACGM:funda,AGV:torsion}) implies that the intersection lattice does not determine the fundamental group of $M(\A)$. Finally, it has been proven by Jiang-Yau~\cite{JiaYau} that the intersection lattice of an arrangement is induced by its topology. 

In order to complete the understanding of these implications between intersection lattice, homotopical type and topological type, we wonder:

\begin{question}
  For a fixed intersection lattice, is the topology of an arrangement determined by the fundamental group or the homotopy-type of its complement?
\end{question}

In other words, are there any $\pi_1$-equivalent or homotopy-equivalent Zariski pairs of line arrangements? Moreover, the particular case of real-complexified arrangements has to be considered too. Indeed, it has recently been proven by Artal, Viu-Sos and the author in~\cite{AGV:torsion}, that the fundamental group of such arrangements is not determined the intersection lattice; solving then a Falk-Randell Problem~\cite{FalkRandell}, and providing an equivalent of Rybnikov's result in the real-complexified case. It is thus logical to look for a complete understanding in this particular case too.\\

In the present paper, we give a negative answer to the previous question. Indeed, in Theorem~\ref{thm:main}, we produce $\pi_1$-equivalent and homotopy-equivalent Zariski pairs from usual Zariski pairs (having some combinatorial properties). We conclude in Corollary~\ref{cor:ZP} remarking that the Zariski pairs of~\cite{ACCM:ZP,Gue:ZP,GueViu,Ryb} all verify the conditions of the previous theorem. This construction can be summarized as follows: we consider the generic union of the two arrangements of a Zariski pair; then we add two lines intersecting in one of the line of the first arrangement for the first case, and in one of the second arrangement for the latter, which are generic with all the other lines. This addition of two extra lines will force any homeomorphism to send the first arrangement on the second and then to create a Zariski pair. The $\pi_1$-equivalence and the homotopy-equivalent are obtained using Theorem~\ref{thm:equivalence} due to~\cite{OkaSak,Wil}.

The conclusions of the papers~\cite{ACCM:ZP} and~\cite{Gue:ZP} use the same argument to remove the ordered condition on their ordered Zariski pairs (ie adding a non-generic line to trivialize the combinatorics automorphism group). In the last section of this paper, we give alternative ends to these papers allowing to obtain $\pi_1$-equivalent and homotopy-equivalent Zariski pairs. The first provided pair is composed of complex line arrangements with 13 lines derived from the Zariski pair obtained by the author in~\cite{Gue:ZP} and is $\pi_1$-equivalent. The latter is formed by real line arrangements composed of 14 lines derived from the ordered Zariski pair produced in~\cite{ACCM:ZP} and is homotopy-equivalent.

%% file: ZP.tex
\subsection{Combinatorics and ordered Zariski pair}\mbox{}

	The \emph{combinatorics} of an arrangement $\A=\{L_1,\dots,L_n\}$ is encoded in the \emph{intersection lattice} (or equivalently in the underlying matroid). This lattice is given by: $\L(\A)=\{\bigcap_{L\in\B} L \neq \emptyset \mid \B\subset\A \}$, and it is ordered by the reverse inclusion. An isomorphism between the intersections lattices of $\A_1$ and $\A_2$ is a bijection between $\A_1$ and $\A_2$ which respect $\L(\A_1)$ and $\L(\A_2)$ together with the reverse inclusion. Such arrangements are called \emph{lattice isomorphic}.
	
	We can add a total order on the line of an arrangement $\A$ and then consider \emph{ordered arrangement} and the associated \emph{ordered combinatorics}. An isomorphism between the intersection lattice of two ordered arrangements is \emph{ordered} if it respect the fixed orders on the arrangements and \emph{non-ordered} otherwise. Notice that if there is an ordered isomorphism between two ordered intersection lattices, then this isomorphism is unique.
 
  \begin{rmk}
      Let $\A=\{L_1,\dots ,L_n\}$ be an ordered arrangement. If no order is precised then we consider the one given by the indices.
  \end{rmk}

\noindent We define an \emph{ordered Zariski pair} as a couple of ordered arrangements $(\A_1,\A_2)$ such that:
  \begin{itemize}
    \item $\A_1$ and $\A_2$ have isomorphic ordered intersection lattice,
    \item Any homeomorphism $\psi$ of $\CC\PP^2$ verifying $\psi(\A_1)=\A_2$ induces a non-ordered isomorphism on the combinatorics of $\A_1$ and $\A_2$. 
  \end{itemize}
	
	Let $\A_1$ and $\A_2$ be two ordered arrangements intersecting generically. We denote the ordered arrangement $\A_1\sqcup\A_2$ by $\A_{1,2}$ where the order is the one induced by those of $\A_1$ and $\A_2$ and such that:
	\begin{equation*}
		\forall L^1\in\A_1, \forall L^2\in\A_2, L^1 < L^2.
	\end{equation*}
	
	\begin{rmk}\label{rmk:same_arr}
		The arrangements $\A_{1,2}$ and $\A_{2,1}$ are not the same ordered arrangement even if they are the same arrangement.
	\end{rmk}

 \subsection{Augmented arrangement and homotopy of the complement}\mbox{}
    
      Let $\A=\{L_1,\dots,L_n\}$ be an arrangement, and let $L$ be a fixed line of $\A$. An \emph{augmented arrangement} of $\A$ along $L$ is an arrangement $\A_L^+=\{L_1,\dots,L_n,L_{n+1},L_{n+2}\}$ such:
      \begin{enumerate}
        \item The lines $L$, $L_{n+1}$ and $L_{n+2}$ are concurrent,
        \item The arrangements $\{L_{n+1},L_{n+2}\}$ and $\A\setminus L$ intersect generically.
      \end{enumerate} 
			
		This construction allows to keep a control on the homotopy of the complement of the augmented arrangement as stated in the following theorem.
		
		\begin{thm}[\cite{OkaSak,Wil}]\label{thm:equivalence}
			Let $\A$ be an ordered line arrangement and $L,L'$ be two lines of $\A$. The arrangements $\A^+_L$ and $\A^+_{L'}$ are $\pi_1$-equivalent. More precisely:
			\begin{equation*}
				\pi_1(M(\A^+_{L})) \simeq \pi_1(M(\A)) \times \FF_2 \simeq\pi_1(M(\A^+_{L'})).
			\end{equation*}
			Futhermore, if $\A$ is a complexified-real arrangement then $\A^+_L$ and $\A^+_{L'}$ are homotopy-equivalent.
		\end{thm}
		
		The first part of the previous theorem is due to Oka and Sakamoto in~\cite{OkaSak}. Notice that this can also be obtained from~\cite{FalkPro} and~\cite{HammLe}. Then, since an augmented arrangement is a 2-generic section of the parallel connection of the given arrangement with a pencil of 3 lines (see~\cite{FalkPro,DenSuc} for more details about parallel connections), the second part of the theorem is given by Williams in~\cite{Wil}.

\subsection{Homotopy-equivalent and $\pi_1$-equivalent Zariski pairs}\label{subsec:existence}\mbox{}
  
	In order to prove our main result, let us recall the notion of \emph{connected arrangement} introduced by Fan~\cite{Fan}. Let $\Sing(\A)$ be the set of all the singular points of $\bigcup_{L\in\A}L$, and let $\S_k(\A)\subset\Sing(\A)$ be the subset of all the singular point of multiplicity $k$. An arrangement $\A$ is connected if the set $\A_{\geq 3}=\bigcup_{L\in\A} L \setminus \S_{\geq 3}(\A)$ is path-connected. Notice that this property is combinatorial.
	
	\begin{thm}\label{thm:main}
		Let $\A_1=\{L_1^1,\dots,L_n^1\}$ and $\A_2=\{L_1^2,\dots,L_n^2\}$ be a Zariski pair, $\phi$ be the order isomorphism between their combinatorics (ie $\phi(L_i^1)=L_i^2$). We fix $k\in\{1,\dots,n\}$ and denote $\ell_j=L_k^j$ (for $j=1,2$). We assume that:
		\begin{enumerate}
			\item\label{C1} The arrangements $\A_1$ and $\A_2$ are connected,
			\item\label{C2} They intersect generically,
			\item\label{C3} For $j=1$ or $2$, any line of $\A_j$ contains at least two points $\S_{\geq 3}(\A_j)$.
		\end{enumerate}
		The arrangements $(\A_{1,2})^+_{\ell_1}$ and $(\A_{2,1})^+_{\ell_2}$ verify the following properties:
		\begin{enumerate}[(I)]
			\item\label{P1} They have isomorphic intersection lattices,
			\item\label{P2} There is no homeomorphism of $\CC\PP^2$ sending $(\A_{1,2})^+_{\ell_1}$ on $(\A_{2,1})^+_{\ell_2}$, 
			\item\label{P3} Their complements are $\pi_1$-equivalent; furthermore, if $\A_1$ and $\A_2$ are real-complexified arrangements then the complements are homotopy-equivalent.
		\end{enumerate}
	\end{thm}
	
	\begin{rmk}
		The conditions~(\ref{C1}) and~(\ref{C3}) are combinatorial, thus if they are verified by $\A_1$ then they are also verified by $\A_2$; and, up to the action of $\PGL_3(\CC)$, condition~(\ref{C2}) is always true.
	\end{rmk}
	
	\begin{proof}\mbox{}
		
		\noindent $\bullet$ (\ref{P1}): The application $\phi^+:(\A_{1,2})^+_{\ell_1}\rightarrow(\A_{2,1})^+_{\ell_2}$ defined below is an (ordered) isomorphism between the intersection lattices.
		\begin{equation*}
			\phi^+ : \left\{ 
			\begin{array}{lll}
				L_i^1 & \longmapsto & L_i^2 \\
				L_i^2 & \longmapsto & L_i^1 \\
				L_{2n+1} & \longmapsto & L_{2n+1} \\
				L_{2n+2} & \longmapsto & L_{2n+2} \\
			\end{array}
			\right.
		\end{equation*}\\
		
		\noindent $\bullet$ (\ref{P2}): We assume that there exists a homeomorphism $\psi^+$ of $\CC\PP^2$ sending $(\A_{1,2})^+_{\ell_1}$ on $(\A_{2,1})^+_{\ell_2}$. By Condition~(\ref{C3}), $L_{2n+1}$ and $L_{2n+2}$ are the only lines of $(\A_{1,2})^+_{\ell_1}$ and $(\A_{2,1})^+_{\ell_2}$ containing a single point of $\S_{\geq 3}$, then $\psi^+(\{L_{2n+1},L_{2n+2}\})=\{L_{2n+1},L_{2n+2}\}$. Thus $\psi^+$ is also a homeomorphism between $\A_{1,2}$ and $\A_{2,1}$. 
		
		By Condition~(\ref{C2}), $(\A_{1,2})_{\geq 3}=(\A_1)_{\geq 3}\sqcup (\A_2)_{\geq 3}$. Furthermore, Condition~(\ref{C1}) induces that the previous decomposition is a decomposition in path-connected components. In particular, this implies that $\psi^+$ fixes or exchanges $\A_1$ and $\A_2$.
		
		Condition~(\ref{C3}) implies that in $\A_{1,2}$ (resp. $\A_{2,1}$) the line $\ell_1$ (resp. $\ell_2$) is the only line containing at least two points of $\S_{\geq 3}(\A_{1,2})$ (resp. $\S_{\geq 3}(\A_{2,1})$) together with the intersection point of the two lines containing a single point of $\S_{\geq 3}$ (by the definition of augmented arrangements). Since $\psi^+$ respects the combinatorics of $(\A_{1,2})^+_{\ell_1}$ and $(\A_{2,1})^+_{\ell_2}$, then $\psi^+(\ell_1)=\ell_2$. This implies, in particular, that $\psi^+$ sends $\A_1$ and $\A_2$, which is impossible since $\A_1$ and $\A_2$ form a Zariski pair.\\
		
		\noindent $\bullet$ (\ref{P3}): By Remark~\ref{rmk:same_arr}, the arrangements $\A_{1,2}$ and $\A_{2,1}$ are the same arrangement. Thus $(\A_{1,2})^+_{\ell_1}$ and $(\A_{2,1})^+_{\ell_2}$ are two augmentation of the same arrangement along different lines. We conclude using Theorem~\ref{thm:equivalence}.		
	\end{proof}
	
  \begin{cor}\label{cor:ZP}
		For a fixed intersection lattice, the topology of an arrangement is not determined by the fundamental group or the homotpy-type of its complement. 
  \end{cor}
	
	\begin{proof}
		The Zariski pairs given in~\cite{Ryb,Gue:ZP} verify Conditions~(\ref{C1})--(\ref{C3}) of Theorem~\ref{thm:main}, then they provide (through the construction of the theorem) $\pi_1$-equivalent Zariski pairs. The ones given in~\cite{ACCM:ZP,GueViu} are real-complexified Zariski pairs. Once again they verify the conditions of Theorem~\ref{thm:main}, thus homotopy-equivalent Zariski pairs exist.
	\end{proof}

%% file: example.tex
  
	In the construction previously given, we obtain $\pi_1$-equivalent and homotopy-equivalent Zariski pairs, but the number of lines needed increase fastly. Indeed, the smallest example provided contains 24 lines. We can produce smaller examples using the ordered Zariski pairs given in~\cite{ACCM:ZP} and~\cite{Gue:ZP} (which are not Zariski pairs). In both papers, the ordered condition is deleted by the addition of a specific line trivializing the automorphism group of the combinatorics. Unfortunately, this operation can also delete the $\pi_1$-equivalence as it has been proven in~\cite{ACGM:funda}. In this section, we give an alternative end at these papers using the notion of augmented arrangement in order to trivialize the automorphism group, and maintain the $\pi_1$-equivalence or the homotopy-equivalence of the arrangements.
  
  These ordered Zariski pairs allow to produce two explicit examples. The pair of~\cite{Gue:ZP} gives rise to an example of $\pi_1$-equivalent Zariski pair with 13 complex lines; while the one of~\cite{ACCM:ZP} provides an example of homotopy-equivalent Zariski pair composed of 14 real lines.
  
\subsection{With complex arrangements}\mbox{}
  
	Let $\R$ be the 10th cyclotomic field, and let $g$ be a generator of the Galois extension. Up to an abuse of notation, the elements of $\R$ are identified with a fixed choice of complex embedding. We define by $\M^+$ (resp. $\M^-$, $\N^+$ and $\N^-$) the arrangement formed by the following 11 lines and where $a=g$ (resp. $a=g^9$, $a=g^3$ and $a=g^7$).
	\begin{equation*}
		\begin{array}{lll}
			L_1 : z=0, & \quad\quad & L_2 : x+y-z=0, \\
			L_3 : x=0, && L_4 : y=0, \\
			L_5 : x-z=0, && L_6 : y-z=0, \\
			L_7 : -a^3x+z=0, && L_8 : y-az=0, \\
			L_9 : (a-1)x-y+z=0, && L_{10} : -a(a-1)x+y+a(a-1)z=0, \\
			L_{11} : -a(a-1)x+y-az=0. && 
		\end{array}
	\end{equation*}
 Remark that the complex conjuagtion sends $\M^+$ on $\M^-$ and $\N^+$ on $\N^-$.
    
    \begin{thm}[Corollary~2.16 and Section~3 of \cite{Gue:ZP}]\label{thm:ExGueZP}
      The pairs $(\M^\pm,\N^\pm)$ and $(\M^\pm,\N^\mp)$ are $\pi_1$-equivalent ordered Zariski pairs.
    \end{thm}
    
    By~\cite[Proposition~2.5]{Gue:ZP}, the automorphism group of the combinatorics of these arrangements is cyclic of order 4. It can be defined as the sub-group of $\Sigma_{11}$ (the symetric group on 11 elements) generated by
    \begin{equation*}
      \sigma = (1\ 3\ 2\ 4)(5\ 6)(7\ 9\ 10\ 8),
    \end{equation*}
    and the action is described by $\sigma\cdot L_i = L_{\sigma(i)}$. In particular, $\sigma$ cyclically permutes the lines $L_1: z=0$, $L_3:x=0$, $L_2:x+y-z=0$ and $L_4:y=0$, in this order. 
    
    In~\cite{Gue:ZP}, the solution used to delete the ordered condition in Theorem~\ref{thm:ExGueZP} is the addition of a line passing through a particular point of multiplicity 4. This allows to trivialize the automorphism group and then to remove the ordered condition. Nevertheless, it has been proven in~\cite{ACGM:funda} that this operation may withdraw the $\pi_1$-equivalence.
    
    Using the augmentation of arrangements defined in Section~\ref{sec:ZP}, we propose an alternative method to remove the ordered condition and conserve the $\pi_1$-equivalence. Indeed, let $\mfM^\pm$ and $\mfN^\pm$ be augmented arrangements of $\M^\pm$ and $\N^\pm$ respectively along $L_1$. We can, for example, consider $\mfM^\pm=\M^\pm \cup \{L_{12},L_{13}\}$ and $\mfN^\pm=\N^\pm \cup \{L_{12},L_{13}\}$, where:
    \begin{equation*}
      L_{12} : x-y+2z=0 \quad\quad\text{and}\quad\quad L_{13} : x-y-2z=0. 
    \end{equation*}
    
    \begin{thm}\label{thm:GueZP}
      The arrangements $\mfM^\pm$ and $\mfN^\pm$ verify the following propositions.
      \begin{enumerate}
        \item The arrangements $\mfM^\pm$ and $\mfN^\pm$ have isomorphic intersection lattice.
        \item There is no homeomorphism of $\CC\PP^2$ sending $\mfM^\pm$ on $\mfN^\pm$.
        \item The fundamental groups of the complements of $\mfM^\pm$ and $\mfN^\pm$ are isomorphic.
      \end{enumerate}
    \end{thm}
    
    \begin{rmk}
      In other words, the couples $(\mfM^\pm,\mfN^\pm)$ and $(\mfM^\pm,\mfN^\mp)$ form $\pi_1$-equivalent Zariski~pairs.
    \end{rmk}
    
    \begin{proof}
      Assume that there exists a homeomorphism $\Psi$ of $\CC\PP^2$ sending $\mfM^\pm$ on $\mfN^\pm$. We denote by $\Phi$ the induced isomorphism on the intersection lattices. Since the lines of $\mfM^\pm\setminus\M^\pm$ and those of $\mfN^\pm\setminus\N^\pm$ are the only lines of $\mfM^\pm$ and of $\mfN^\pm$ containing a single point of $\S_{\geq 3}$, then we have
      \begin{equation*}
        \Psi(\M^\pm)=\N^\pm\quad \text{and}\quad \Phi(\M^\pm)=\N^\pm. 
      \end{equation*}
      We denote by $\phi$ the isomorphism given by the restriction of $\Phi$ to the intersection lattices of~$\M$ and~$\N$.
      
      By Theorem~\ref{thm:ExGueZP}, $\phi$ cannot respect the order on $\M^\pm$ and $\N^\pm$ since it is induced by an homeomorphism. The line $L_1$ is the only one of $\mfM^\pm$ (resp. $\mfN^\pm$) containing the intersection point of the two lines containing a single point of $\S_{\geq 3}$ (ie $L_{12}$ and $L_{13}$). Thus $L_1$ is fixed by $\Phi$ and as a consequence by $\phi$. In particular, this implies that $\phi$ fixes the lines $L_1$, $L_2$, $L_3$ and $L_4$. It follows that $\phi$ should be the ordered isomorphism, which is impossible. 
      
      Since $\M^\pm$ and $\N^\pm$ are $\pi_1$-equivalent arrangements, by Theorem~\ref{thm:equivalence}, it is still the case for $\mfM^\pm$ and $\mfN^\pm$. The lattice isomorphism comes from the one of $\M^\pm$ and $\N^\pm$ together with the construction of augmented arrangements.
    \end{proof}
    
\subsection{With real arrangements}\mbox{}
  
  In order to construct a smaller homotopy-equivalent Zariski pair, we apply a similar argument as previously to the example of Artal-Carmona-Cogolludo-Marco~\cite{ACCM:ZP}. Unfortunately, using a single augmentation is not enough to fix all the automorphisms of the combinatorics as previously done. This problem can be avoided using two successive augmentations.
  
  Let $a$ be a root of $X^2+X-1$, and consider the arrangements $\M$ and $\N$ formed by the 10~lines:
  \begin{equation*}
    \begin{array}{lcl}
      M_1 : z=0, &\quad\quad\quad& L_1 : x-y, \\
      M_2 : x=0, && L_2 : ax-y-az=0, \\
      M_3 : x-z=0, && L_3 : ax-y+z=0, \\
      M_4 : x+(a+1)z=0, && L_4 : y-z=0, \\
      M_5 : x-(a+2)z=0, && L_5 : y=0.
    \end{array}
  \end{equation*}
  
  \begin{thm}[Remark~2.8 and Theorem~4.19 of~\cite{ACCM:ZP}]\label{thm:ACCMoZP}
    The arrangements $\M$ and $\N$ form a homotopy-equivalent ordered Zariski pair.
  \end{thm}
  
  By~\cite[Lemma~2.9]{ACCM:ZP}, the automorphsim group of the combinatorics of $\M$ (and $\N$ too) is isomorphic to the sub-group of $\Sigma_5$ (the symetric group on 5 elements) generated by:
  \begin{equation*}
    \sigma_1=(1,2,3,4,5)\quad \text{ and } \sigma_2=(2,4,5,3).
  \end{equation*}
  More precisely, it is the semi-direct product of $\langle \sigma_1 \rangle$ and $\langle \sigma_2 \rangle$.
  
  The action of the generators can be viewed as an action on the five lines $M_i$, and is given by $\sigma_j \cdot M_i = M_{\sigma_j(i)}$. The idea is to trivialize this automorphism group by augmentations of these arrangements. The problem is the following: if we fix one of the line $M_i$ with an augmentation, then the automorphism of the obtained combinatorics is still not trivial (indeed the stabilisator of any line is never the whole group). Thus, we need to consider an additional augmentation. 

  Let $\mfM_{M_1,L_5}^+$ (resp. $\mfN_{M_1,L_5}^+$) be the arrangement arising from two augmentations of $\M$ (resp. $\N$), along $M_1$ and $L_5$. Furthermore, we can assume that the four added lines are defined by real linear forms, in such way that $\mfM_{M_1,L_5}^+$ and $\mfN_{M_1,L_5}^+$ are real-complexified arrangements. For example, we can consider the arrangement $\mfM_{M_1,L_5}^+=\M\cup\{D_1,D_2,D_3,D_4\}$ and $\mfN_{M_1,L_5}^+=\N\cup\{D_1,D_2,D_3,D_4\}$, where
  \begin{equation*}
    \begin{array}{ccc}
      D_1 : x+y+z=0 & \quad\quad\text{and}\quad\quad & D_2 : x+y+2z=0, \\
      D_3 : x+3y-5z=0 & \quad\quad\text{and}\quad\quad & D_4 : x-3y-5z=0.
    \end{array}
  \end{equation*}
  These arrangements are well augmented arrangements of $\M$ and $\N$ since the lines $M_1$, $D_1$ and $D_2$ (resp. $L_5$, $D_3$ and $D_4$) are concurrent; and $D_1,\, D_2$ (resp. $D_3,\, D_4$) are generic with all the other lines.
  
  \begin{lem}\label{lem:ACCM_automorphism}
      Any isomorphism between $\L(\mfM_{M_1,L_5}^+)$ and $\L(\mfN_{M_1,L_5}^+)$ restricts to an ordered isomorphism between $\L(\M)$ and $\L(\N)$.
  \end{lem}
  
  \begin{proof}
    Let $\phi^+$ be an isomorphism between $\L(\mfM_{M_1,L_5}^+)$ and $\L(\mfN_{M_1,L_5}^+)$. The four lines of $\mfM_{M_1,L_5}^+\setminus\M$ and the ones of $\mfN_{M_1,L_5}^+\setminus\N$ (that is $D_1,\dots,D_4$) contain only one point of $\S_{\geq 3}$ (in opposition with all the others which contain at least two points). This implies that $\phi^+$ restricts to an isomorphism $\phi$ between $\L(\M)$ and $\L(\N)$.
    
    Remark that $M_1$ and $L_5$ are the only lines containing one of the intersection points of the four additional lines. Furthermore only $M_1$ contains a quintuple point. Then $M_1$ and $L_5$ are fixed by $\phi^+$ and as a consequence by $\phi$. In $\M$ and $\N$, the line $L_1$ (defined by $x-y=0$) is the only one intersecting $M_1$ in a double point, it is thus fixed by $\phi$. Since it fixes $L_5$ and $L_1$ then it also fixes $M_2$ because they intersect in a triple point. 
    
   Using the description of the automorphism group of the combinatorics of $\M$ and $\N$ as a sub-group of $\Sigma_5$ previously given and the fact that $\phi$ fixes $M_1$ and $M_2$, we deduce that $\phi$ is ordered.
  \end{proof}
  
  \begin{thm}\label{thm:ACCM_ZP}
    The real-complexified arrangements $\mfM_{M_1,L_5}^+$ and $\mfN_{M_1,L_5}^+$ verify the following propositions:
       \begin{enumerate}
        \item The arrangements $\mfM_{M_1,L_5}^+$ and $\mfN_{M_1,L_5}^+$ have isomorphic intersection lattices.
        \item There is no homeomorphism of $\CC\PP^2$ sending $\mfM_{M_1,L_5}^+$ on $\mfN_{M_1,L_5}^+$.
        \item The complements of $\mfM_{M_1,L_5}^+$ and $\mfN_{M_1,L_5}^+$ are homotopy-equivalent.
      \end{enumerate}
   \end{thm}
  
  \begin{proof}
    Let $\psi^+$ be a homeomorphism of $\CC\PP^2$ sending $\mfM_{M_1,L_5}^+$ on $\mfN_{M_1,L_5}^+$. We denote by $\phi^+$ the induced isomorphism on the intersection lattices. By Lemma~\ref{lem:ACCM_automorphism}, $\phi^+$ restricts into an ordered isomorphism between $\L(\M)$ and $\L(\N)$ which is in conflict with Theorem~\ref{thm:ACCMoZP}.
    
    We conclude noticing that the construction of $\mfM_{M_1,L_5}^+$ and $\mfN_{M_1,L_5}^+$ compels that they are lattice isomorphic; and the homotopy-equivalence is a direct consequence of Theorem~\ref{thm:ACCMoZP} and Theorem~\ref{thm:equivalence}.
  \end{proof}
	